\documentclass[11pt]{amsart}
\usepackage{amsmath,amssymb,amsfonts,amsthm,amscd,indentfirst,color}

\numberwithin{equation}{section}

\newtheorem{prop}{Proposition}[section]
\newtheorem{theo}[prop]{Theorem}
\newtheorem{lemm}[prop]{Lemma}

\newtheorem{rema}[prop]{Remark}

\def\di{\displaystyle}
\def\Dint{\displaystyle\int}

\def\O{\Omega}
\def\o{\omega}
\def\n{\nabla}
\def\P{\partial}
\def\Dint{\displaystyle\int}
\def\di{\displaystyle}
\def\<{\langle}
\def\>{\rangle}

\begin{document}
\title[Warped product spaces]{ A volume preserving flow and the isoperimetric problem in warped product spaces}

\author{Pengfei Guan, Junfang Li, and Mu-Tao Wang}

\address{Department of Mathematics and Statistics\\
McGill University\\
Montreal, Quebec H3A 0B9, Canada}
\email{guan@math.mcgill.ca}

\address{Department of Mathematics\\
University of Alabama at Birmingham\\
Birmingham, AL 35294}
\email{jfli@uab.edu}

\address{
   Department of Mathematics\\
Columbia University\\
2990 Broadway, New York, NY 10027
   }
\email{mtwang@math.columbia.edu}
\thanks{Research of first author was supported in part by an NSERC Discovery Grant.
This material is based upon work supported by the National Science
Foundation under Grant Numbers DMS 1405152 (Mu-Tao Wang).}

\begin{abstract}
    In this article, we continue the work in \cite{GL} and study a normalized hypersurface flow in the more general ambient setting of warped product spaces. This flow preserves the volume of the bounded domain enclosed by a graphical hypersurface, and monotonically decreases the hypersurface area. As an application, the isoperimetric problem in warped product spaces is solved for such domains.\end{abstract}

\date{\today}
\maketitle

\section{Introduction}
Let $(\mathbf B^n,\tilde g)$ be a closed Riemannian manifold. Let $\phi=\phi(r)$ be a smooth positive function defined on the interval $[r_0,\bar r]$ for some $r_0 <\bar r $.  We consider a Riemannian manifold $(\mathbf N^{n+1},\bar g)$ (possibly with boundary) with the warped product structure,
\begin{equation}
  \bar g=dr^2+\phi^2\tilde g,      \,\,\, r\in [r_0, \bar{r}]
  \label{warped metric}
\end{equation}
where $\tilde g$ is the metric of the manifold $\mathbf B^n$.  $\mathbf N^{n+1}$ is naturally equipped with a conformal Killing field $X=\phi (r)\partial_r$.  Let $M$ be a smooth closed embedded hypersurface in  $\mathbf N^{n+1}$, which is parametrized by an embedding $F_0$. We consider the following evolution equation for a family of embeddings of hypersurfaces with $F_0$ as an initial data, i.e. $F(\cdot, t)=F_0$:

\begin{equation}
    \frac{\partial F}{\partial t}=(n\phi'-uH)\nu,
    \label{eflow0}
  \end{equation}
  where $\nu$ is the outward unit normal vector field,  $H$ is the mean curvature, and $u=\<X,\nu\>$ is the support function of the hypersurface defined by $F(\cdot, t)$. A hypersurface $M$ is said to be  {\it graphical} if it is defined by $r=\rho(p), p\in \mathbf B^n$ for a smooth function $\rho$ on $\mathbf B^n$. When $(\mathbf B^n,\tilde g)$ is the standard unit sphere $\mathbb S^n$ in $\mathbb R^{n+1}$ and $\phi(r)=\sin(r), r, \sinh(r)$,  $(\mathbf N^{n+1},\bar g)$ represents $\mathbb S^{n+1}, \mathbb R^{n+1}, \mathbb H^{n+1}$ respectively. In these special cases, flow (\ref{eflow0}) was studied in \cite{GL} in connection with the isoperimetric problem. In this article, we consider (\ref{eflow0}) in the more general ambient setting of warped product spaces.

\medskip

Below are our main theorems.
\begin{theo}   \label{main thm}
    Let $M_0$ be a smooth graphical hypersurface in $(\mathbf N^{n+1},\bar g)$ with $n\ge2$ and $\bar{g}$ in \eqref{warped metric}. If $\phi(r)$ and  $\tilde g$ satisfy the following conditions:
    \begin{equation}\label{condition}
    \begin{aligned}
       & \tilde{Ric}\geq (n-1)K \tilde{g},\\
       & 0\le (\phi')^2-\phi''\phi\le K \text{ on } [r_0, \bar{r}]
    \end{aligned}
    \end{equation}
where $K>0$ is a constant and $\tilde{Ric}$ is the Ricci curvature of $\tilde{g}$, then the evolution equation \eqref{eflow0} with $M_0$ as the initial data
has a smooth solution for $t\in [0,\infty)$. Moreover, the solution hypersurfaces converge exponentially to a level set of $r$ as $t\rightarrow \infty$.
\end{theo}

As an application, we obtain a solution to the isoperimetric problem for warped product spaces.   Let $S(r)$ be a level set of $r$ and $B(r)$ be the bounded domain enclosed by $S(r)$ and $S(r_0)$. The volume of $B(r)$ and surface area of $S(r)$, both positive functions of $r$, are denoted as $V(r)$ and $A(r)$, respectively. Note that $V=V(r)$ is strictly increasing function of $r$.  Consider the single variable function $\xi(x)$ that satisfies
\begin{align}
    A(r)=\xi(V (r)),
    \label{defining function}
\end{align}
for any $r\in[ r_0,\bar r]$.  The function $\xi(x)$ is well-defined.

\begin{theo}\label{main thm2}
    Let $\O\subset \mathbf N^{n+1}$ be a domain bounded by a smooth graphical hypersurface $M$ and $S(r_0)$.    We assume $\phi(r)$ and  $\tilde g$ satisfy the conditions \eqref{condition} in Theorem \ref{main thm}, then
    \begin{align}
        Area(M)\ge \xi(Vol(\O)),
        \label{iso prob}
    \end{align}
    where $Area(M)$ is the area of $M$ and $Vol(\O)$ is the volume of $\Omega$, and function $\xi$ is defined in (\ref{defining function}). If, in addition to \eqref{condition}, either  $(\phi')^2-\phi''\phi< K$ or $\tilde{Ric}> (n-1)K \tilde{g} \text{ on } [r_0, \bar{r}]$ then  ``=" is attained in \eqref{iso prob} if and only if $M$ is a level set of $r$.
\end{theo}

Some remarks are in order.
\begin{rema}\label{rem1}
    \begin{enumerate}
        \item [(i)] The upper bound condition $ (\phi')^2-\phi''\phi\le K$ is needed for the monotonicity property of the flow, see Theorem \ref{thm mono hyperbolic}. Indeed, the condition in this setting implies the corresponding level set of $r$ is a stable CMC, which locally minimizes areas subject to the constraint of fixing enclosed volumes. More details of these conditions can be found in Section \ref{stability}.
                \item [(ii)] The lower bound condition $ (\phi')^2-\phi''\phi\ge0$ is needed for the gradient estimate and this condition is closely related to the notion of ``photon sphere'' in general relativity, see more details in Section \ref{photon}.
        \item[(iii)] The function $A(r)$ is given explicitly by $A(r)=\phi^n (r) Area({\bf B}^n)$ and $V(r)$ is characterized by the ODE: $\frac{dV}{dr}=A(r), V(r_0)=0$.
        To determine the function $\xi$, one can first solve $r$ in terms of $V$ and then plug into the formula of $A(r)$. For example, when $n=1$ and $\phi(r)=\sin (r)$, we deduce that $A(r)=\sqrt{V(r)(4\pi-V(r))}$ or $\xi(x)=\sqrt{x(4\pi-x)}$. When $n=1$ and $\phi(r)=\sinh (r) $, we deduce that $A(r)=\sqrt{ V(r)(4\pi+V(r))}$ or $\xi(x)=\sqrt{x (4\pi+x)}$.
        \item[(iv)] The stability condition is local in nature and concerns only the geometry of the submanifold, while the isoperimetric problem concerns the global geometry of the ambient manifold. There are interesting conjectures about the isoperimetric problem,  see for example \cite{H, M}
        \end{enumerate}\end{rema}

The hypersurface flow (\ref{eflow0}) is a local flow that preserves enclosed volume. This seems to be a novel feature compared with the known flows in the literature. More specifically, equation (\ref{eflow0}) is a pointwise defined parabolic PDE and it preserves the enclosed volume along the flow. To the authors' knowledge, most hypersurface flows in the literature are either local that do not preserve integral geometric quantities, such as volume, surface area, etc., or globally defined that involve some integral terms. For example, in Huisken's famous work on mean curvature flow \cite{H1}, the original local flow is not volume preserving. If one rescales the hypersurface so that the volume is preserved along the flow, then an extra integral term which involves total squared mean curvature and surface area has to be included.  On the other hand, in another paper of Huisken \cite{H3}, a volume preserving mean curvature flow was discussed. The definition of this flow already contains a global quantity related to total mean curvature and surface area. More details of these comparisons can be found in previous work \cite{GL}.

Another advantage of the flow (\ref{eflow0}) is that the existence and exponential convergence  do not depend on any convexity condition of the domain.  This also seems to be a surprising property for hypersurface flows.

The rest of the paper is organized as follows.
In section \ref{hyper}, we discuss hypersurfaces in warped product spaces, and prove a Minkowski identity and monotonic properties along the normalized flow. In section \ref{c1est}, we convert the flow into a parabolic PDE for a graphical hypersurface and prove the $C^0$ estimate, the main gradient estimate and the exponential convergence of the flow under conditions in Theorem \ref{main thm}. In section \ref{uH}, we derive evolution equations for support function and for mean curvature, then obtain bounds for these geometric quantities.  Theorem 1.1 and Theorem 1.2 are proved in section \ref{proofs}. In section \ref{photon}, we discuss the conditions imposed  on the warping function $\phi$. In the last section, we discuss the convergence of the flow in the case when $K=0$ in \eqref{condition}.

\section{A Minkowski identity and the monotonicity}\label{hyper}

Throughout this paper, we use Einstein convention for repeated indexes. We use $\bar g$, $g$, and $\tilde g$ to denote the metrics of the ambient warped product space $\mathbf N^{n+1}$, hypersurface $M^n$, and the base $\mathbf B^n$ respectively. Consequently, we use $\bar \n$, $\n$, and $\tilde\n$ to denote gradient with respect to the metrics $\bar g$, $g$, and $\tilde g$ respectively. Similarly, we have notations such as Laplacian $\bar \Delta$, $\Delta$, and $\tilde \Delta$ in different contexts.

Let $(\mathbf N^{n+1},\bar g)$ be a Riemannian manifold with warped product structure,
\begin{equation}
  \bar g= ds^2=dr^2+\phi^2\tilde g,
  \label{}
\end{equation}
where $\tilde g$ is the metric of the base manifold $\mathbf B^n$ and $\phi=\phi(r)$ is a smooth positive function on $(r_0, \bar r)$ for some $\bar r\le\infty$.

\begin{lemm}
    \label{hessian lemma}
Let $Y=\eta(r)\P_r$ be a vector field. Then the Lie derivative of $Y$ is given by
$\mathcal L_Y \bar g= 2\eta' (dr^2 +\frac{\eta\phi'}{\eta'\phi}\phi^2\tilde g)$. In particular, the vector field $X=\phi(r)\P_r$ is a conformal Killing field, i.e., $\mathcal L_X \bar g=2\phi'(r)\bar g$. Moreover, $D_iX_j = \phi'(r) \bar g_{ij} $.
\end{lemm}

\begin{proof}
  Recall the Lie derivatives for differential forms are
  \[
    \begin{array}[]{rll}
      \mathcal L_Y dy^{\beta} = & \frac{\P \eta^\beta}{\P y^{\alpha}}dy^{\alpha}\\
      \mathcal L_Y f = & Y(f).
    \end{array}
    \]
Thus,
  \[
    \begin{array}[]{rll}
      \mathcal L_Y dr = & \eta'(r)dr\\
      \mathcal L_Y dr\otimes dr = & 2\eta'(r)dr\otimes dr.\\
    \end{array}
    \]
    and
  \[
    \begin{array}[]{rll}
      \mathcal L_Y \phi^2\tilde g = &2\phi\phi'\eta\tilde g\\
      \mathcal L_Y \bar g = & 2\eta'(r)dr\otimes dr+2\phi\phi'\eta\tilde g \\
      =& 2\eta'(r)(dr\otimes dr +\frac{\phi'\eta}{\phi\eta'}\phi^2\tilde g).
    \end{array}
    \]
    Let $\eta=\phi$, then $Y=X$. The second part of the lemma follows immediately.
\end{proof}

By direct computations, see for example \cite{Be} and \cite{Br}, we have the Ricci tensor with respect to the metric $\bar g$.
\begin{lemm}
    \label{Ricci lemma}
  The Ricci curvature tensor of $(\mathbf N^{n+1},\bar g)$ is given by

  \begin{equation}
  \begin{array}[]{rll}
    \bar Ric=& -n\frac{\phi''}{\phi}dr^2-[(n-1)\phi'^2)+\phi\phi'']\tilde g+\tilde Ric.\\
  \end{array}
  \label{}
\end{equation}
In particular, for any $K\in \mathbb R$,
\begin{equation}
  \begin{array}[]{rll}
    \bar Ric=& -n\frac{\phi''}{\phi}dr^2+[(n-1)(K-\phi'^2)-\phi\phi'']\tilde g+\tilde Ric-(n-1)K\tilde g.\\
  \end{array}
  \label{}
\end{equation}
\end{lemm}

\medskip

Let $M^n\subset \mathbf N^{n+1}$ be a smooth hypersurface in the warped product space. Under local coordinates on $M^{n}$, denote by $g_{ij}$,  $h_{ij}$,  $h^i_j= g^{ik} h_{kj}$, and $H=h^i_i$, the induced metric,  the second fundamental form, the Weingarten tensor, and the mean curvature respectively for $i,j=1,\cdots,n$. We also let $\sigma_l$ denote the $l$-th elementary symmetric functions of the principal curvatures, i.e., the eigenvalues of the Weingarten tensor for $1\le l\le n$.

We need the following lemma.

\begin{lemm}
  \label{cofactor lemma}
  Let $\sigma_2^{ij}=\frac{\P\sigma_2}{\P h_{ij}}=H g^{ij}-h^{ij}$ be the cofactor tensor. Then the trace of its covariant derivative is
  \begin{equation}
  \sigma_2^{ij}(h)_j=-\bar R_{i\nu},
    \label{ric}
  \end{equation}
  where $\nu$ is the unit outward normal of the hypersurface.
\end{lemm}

\begin{proof}
    In this proof, we will not use Einstein convention temporarily. For convenience, we use orthonormal coordinates and do not distinguish upper and lower indexes. By definition of $\sigma_l$, see e.g. \cite{R}, we have
  \begin{equation}
    \begin{array}[]{rll}
      \sum_{j=1}^n \sigma_2^{ij}(h)_j=& \sigma_2^{ii}(h)_i+\sum_{j\neq i}\sigma_2^{ij}(h)_j\\
      =&\sum_{l=1}^n h_{ll,i}-h_{ii,i}-\sum_{j\neq i}h_{ji,j}\\
    \end{array}
    \label{}
  \end{equation}
By  the Codazzi equation, we obtain
\begin{equation}
  h_{il,j}-h_{jl,i}= \langle \bar R(\P_j,\P_i)\nu,\P_l\rangle =\bar R_{ji\nu l},
  \label{}
\end{equation}
and
\begin{equation}
 \sum_{l\neq i}[ h_{il,l}-h_{ll,i}]=\bar R_{i\nu }.
  \label{Ricci iden}
\end{equation}

Thus,
  \begin{equation}
    \begin{array}[]{rll}
      \sum_{j=1}^n \sigma_2^{ij}(h)_j=&-\bar R_{i\nu}    \end{array}
    \label{}
  \end{equation}

\end{proof}

The following lemma is well-known, for example, can be found in \cite{Br} which follows from Lemma \ref{hessian lemma} and the Gauss equation directly.
\begin{prop}
\label{prop hessian}
  Let $X=\phi(r)\P_r$ be the conformal vector field and $\Phi'(r)=\phi(r)$. Then on a hypersurface $M\subset \mathbf N^{n+1}$,
  \begin{equation}
  \begin{aligned}
      \Phi_{ij}& = \phi'(r)g_{ij}-u h_{ij}\\
       \Delta \Phi &= n\phi'(r)-Hu,
  \end{aligned}
  \end{equation}
  where $u=\langle X,\nu\rangle $, $\Phi_{ij}$ is the Hessian of the function $\Phi$,  $\Delta\Phi$ is the Laplacian of the function $\Phi$, both with respect to the induced metric $g$ on $M$.\end{prop}

Now we derive a Minkowski identity.
\begin{lemm}
  \label{lemma Minkowski identity}
    Let $X=\phi(r)\P_r$ be the conformal vector field and $\Phi'(r)=\phi(r)$. Then on a hypersurface $M\subset\mathbf  N^{n+1}$,
  \begin{equation}
   (n-1) \Dint_M\phi' \sigma_1d\mu = 2 \Dint_M\sigma_2 ud\mu+\Dint_M \bar R_{i\nu}\Phi_i,
    \label{}
  \end{equation}
  where $u=\langle X,\nu\rangle $.
\end{lemm}
\begin{proof}
  Applying Proposition \ref{prop hessian} and  contracting the cofactor tensor $\sigma_2^{ij}$ with the hessian of $\Phi$, we have
  \begin{equation}
    \sigma_2^{ij}\Phi_{ij}=(n-1) \phi' \sigma_1 - 2\sigma_2 u.
    \label{sigma 2 lemma equ 1}
  \end{equation}
  Integrate equation (\ref{sigma 2 lemma equ 1}) over $M$ and after integration by parts, we have
  \begin{equation}
    \begin{array}[]{rll}
    & (n-1)\Dint_M \phi' \sigma_1d\mu - 2\Dint_M\sigma_2 ud\mu\\
     =&\Dint_M\sigma_2^{ij}\Phi_{ij}d\mu\\
    =&-\Dint_M\sigma_2^{ij}(h)_j\Phi_i\\
    =&\Dint_M \bar R_{i\nu}\Phi_i,
     \end{array}
    \label{}
  \end{equation}
  where the last inequality follows from \eqref{ric}.
\end{proof}

\medskip

Let $M(t)$ be a smooth family of closed hypersurfaces in $\mathbf N^{n+1}$. Let $F(\cdot,t)$ denote a point on $M(t)$. We consider the flow \eqref{eflow0} in $(\mathbf N^{n+1},\bar{g} )$ where $\bar{g}$ is given as in \eqref{warped metric}.

\begin{prop}\label{prop2.1} Under flow $\partial_t F = f\nu$ of closed hypersurfaces in a Riemannian
manifold, suppose $\Omega_t$ is the domain enclosed by the evolving
hypersurface $M(t)$ and a fixed hypersurface , we have the following evolution equations.

\begin{equation}\label{}
\begin{array}{rll}
\partial_t g_{ij} &=& 2fh_{ij}\\
\partial_t h_{ij} &=& -\nabla_i\nabla_j f + f (h^2)_{ij}-fR_{\nu ij \nu}\\
\partial_t h^j_i  &=& -g^{jk}\nabla_k\nabla_j f - g^{jk}f (h^2)_{ki}-fg^{jk}R_{\nu ik \nu}\\
\end{array}
\end{equation}

Moreover, we have
\begin{align*}
    A'(t)=& \Dint_M fHd\mu_g,\\
V'(t)=&\Dint_M fd\mu_g,
\end{align*} where $A(t)$ is the area of $M(t)$ and $V(t)$ is the volume of $\Omega_t$.
\end{prop}

Using Proposition \ref{prop2.1}, we obtain the following monotonicity formulae.

\begin{theo}\label{thm mono hyperbolic}
Let $M(t)$ be a smooth one-parameter family of closed hypersurface
in $\mathbf N^{n+1}$ with $M(0)=\partial \Omega$ which solves the parabolic equations (\ref{eflow0}) on $[0,T)$. We assume $M(t)$ are graphical  hypersurfaces. If $K-\phi'^2+\phi\phi''\ge 0$ and $\tilde{Ric}\geq (n-1)K \tilde{g}$, then the enclosed volume is a constant and surface area is non-increasing along the flow.
\end{theo}
\begin{proof}  The proof is a consequence of Proposition \ref{prop2.1} and the Minkowski identity from Lemma \ref{lemma Minkowski identity}.

  \begin{align}
         V'(t)=&\Dint (n\phi' - Hu)d\mu_g =0\nonumber\\
  \nonumber\\
 A'(t)=& \Dint (n\phi'-Hu)Hd\mu_g\nonumber\\
 =& \Dint (n\phi' H-\frac{2n}{n-1}\sigma_2u)d\mu_g+\Dint (\frac{2n}{n-1}\sigma_2-H^2)ud\mu_g\nonumber\\
 =& \Dint \bar n R_{i\nu}\n_i\Phi d\mu_g+\Dint (\frac{2n}{n-1}\sigma_2-H^2)ud\mu_g\nonumber\\
 \le & 0,
      \label{mono 2 hyperbolic}
  \end{align}
  where we have used Lemma \ref{lemma Minkowski identity}, (\ref{Ric normal}), and Newton-McLaurin inequality.
\end{proof}

Note that above proof fails $n=1$. This case is treated by different argument in \cite{Cant}.

\section{Graphical hypersurface and $C^1$ estimate}\label{c1est}
We now focus only on those hypersurfaces that are graphical.
Let  $M$ be the graph of a smooth and positive function $\rho$ on $\mathbf B^n$. Let $\P_1, \cdots, \P_n$ be a local frame along $M$ and $\P_\rho$ be the vector field along radial direction.  For simplicity, all the covariant derivatives  are with respect to the metric $\tilde g_{ij}$ and  denoted as $\tilde\nabla$ when there is no confusion in the context.

Denote \[\omega:=\sqrt{\phi^2+|\tilde\nabla\rho|^2},\] then the outward unit normal is $\nu= \frac{\phi}{\omega}(1,-\frac{\rho_1}{\phi^2},\cdots, -\frac{\rho_n}{\phi^2})$. The support function, induced metric, inverse metric matrix, second fundamental form can be expressed as follows.
\begin{equation}\label{tensors:rho}
    \begin{array}{rll}
    u=&\frac{\phi^2}{\omega}\\
    \rho^i=& \tilde g^{il}\rho_l\\
    g_{ij}=& \phi^2\tilde{g}_{ij}+\rho_i\rho_j,\quad g^{ij}=\frac{1}{\phi^2}(\tilde{g}^{ij}-\omega^{-2} \rho^i\rho^j)\\
    h_{ij}=&\omega^{-1}(-\phi\tilde\nabla_i\tilde\nabla_j\rho+2\phi'\rho_i\rho_j+\phi^2\phi'\tilde{g}_{ij})\\
    h^i_j=& \frac{1}{\phi^2 \omega}(\tilde{g}^{ik}-\omega^{-2} \rho^i\rho^k)(-\phi\tilde\nabla_k\tilde\nabla_j\rho+2\phi'\rho_k\rho_j+\phi^2\phi'\tilde{g}_{kj})
    \end{array}
\end{equation}
where all the covariant derivatives $\tilde\nabla$ and $\rho_i$ are w.r.t. the base metric $\tilde g_{ij}$. \\

For convenience, we let
\begin{equation}\label{tensors:rho 2}
    \begin{array}{rll}
      b_{ij}=& -\phi\o^2\rho_{ij}+\phi\rho_i(\frac{|\tilde\n\rho|^2}{2})_j+\phi'\phi^2\rho_i\rho_j+\phi'\phi^2\o^2\tilde{g}_{ij}\\
    h^i_j=& \frac{1}{\phi^2\omega^3}b^i_j\\
    \tilde H=& b^i_i=-\phi\o^2\tilde\Delta\rho+\phi\rho^i(\frac{|\tilde\n\rho|^2}{2})_i+\phi'\phi^2|\tilde\n\rho|^2+n\phi'\phi^2\o^2\\
    \end{array}
\end{equation}
Thus, $\tilde H=\phi^2\o^3 H$ and $H=  \frac{1}{\phi^2\o^3}\tilde H$

By direct computations, we have
\begin{lemm}
    Let $M^n\subset\mathbf  N^{n+1}$ be a graphical hypersurface, which is defined by a function $r=\rho(p)$, $p\in \mathbf B^n$. Then
\begin{align}
    \bar R_{\nu \nu}+n\frac{\phi''}{\phi}=&(n-1)\frac{K-\phi'^2+\phi\phi''}{\phi^2}\frac{|\tilde\n \rho|^2}{\o^2}+(\tilde R_{ij}-(n-1)K\tilde g_{ij})\frac{\rho_i\rho_j}{\phi^2\o^2},\nonumber\\
    g^{ij}\bar R_{\nu e_i}\n_j\Phi=&-(n-1)(K-\phi'^2+\phi\phi'')\frac{|\tilde\n\rho|^2}{\o^3} -(\tilde R_{ij}-(n-1)K\tilde g_{ij})\frac{\rho_i\rho_j}{\o^3},\nonumber\\
        \label{Ric normal}
    \end{align}
 where $ \nu = \frac{\phi}{\sqrt{\phi^2+|\tilde\n \rho|^2}}(\P_r-\frac{\rho_i}{\phi^2}\P_i)$ is the unit outward normal vector
and $ e_i=\rho_i\P_r+\P_i$ are the tangent vector fields.
\end{lemm}
\begin{proof}
    Using Lemma \ref{Ricci lemma}, we first compute
    \begin{align}
     \bar R_{\nu \nu}=&-(n-1)(K-\phi'^2+\phi\phi'')\frac{u^2}{\phi^4}+((n-1)(K-\phi'^2)-\phi\phi'')\frac{1}{\phi^2}\nonumber\\
     &+(\tilde R_{ij}-(n-1)K\tilde g_{ij})\frac{\rho_i\rho_j}{\phi^2\o^2}.
    \label{Rnunu equ1}
    \end{align}
The first identity of the lemma follows immediately after simplifications.

Using Lemma \ref{Ricci lemma} again, we have
    \begin{align}
        \bar R_{\nu e_i}=&-(n-1)\frac{K-\phi'^2+\phi\phi''}{\phi\o}\rho_i -(\tilde R_{ik}-(n-1)K\tilde g_{ik})\frac{\rho_k}{\phi\o}.
    \label{Rnunu equ2}
    \end{align}
    Combining (\ref{tensors:rho}) and (\ref{Rnunu equ2}), we finish the proof of the second identity.
\end{proof}

\medskip

We now consider the flow equation (\ref{eflow0}) of graphical hypersurfaces in $\mathbf N^{n+1}$. It is known that if a closed hypersurface is graphical and satisfies
\[
\P_t F= f\nu,
\]
then the evolution of the scalar function $\rho=\rho(F(z,t),t)$   satisfies\[
\P_t\rho = f\frac{\omega}{\phi}.
\]

Thus it suffices to consider the following parabolic initial value problem on $\mathbf B^n$,
\begin{equation}
  \left\{
  \begin{array}[]{rll}
    \P_t\rho =& (n\phi'-Hu)\frac{\omega}{\phi},  \rho=\rho(p, t)  \text{ for } (p,t)\in \mathbf B^n\times [0,\infty)\\
    \rho(\cdot, 0)=&\rho_0,
  \end{array}
  \right.
  \label{ivp}
\end{equation}
where $\rho_0$ is the radial function of the initial hypersurface.

We next show that the radial function $\rho$ is uniformly bounded from above and below.
\begin{prop}
  \label{rho bounded}
  Let $M_0$ be a graphical hypersurface defined by function $\rho_0$ in $\mathbf N^{n+1}$. If $\rho(p,t)$ solves the initial value problem \eqref{ivp}, then for any $(p,t)\in \mathbf B^n\times [0, T)$,
\[\min_{p\in \mathbf B} \rho(x,0)\le \rho(p, t)\le \max_{p\in \mathbf B} \rho(p,0).\]
\end{prop}
\begin{proof}
  At critical points of $\rho$, the following conditions hold,
 \[\tilde\nabla \rho=0,  \omega =\phi.\] It follows from (\ref{tensors:rho}) that, at critical points of $\rho$, $\tilde H=-\phi^3\tilde\Delta \rho+n\phi'\phi^4$. Together with (\ref{ivp}), at critical points,
 \[
   \rho_t=\frac{1}{\phi}\tilde\Delta \rho.
   \]
By the standard maximum principle, this proves the uniform upper and lower bounds for $\rho$.
\end{proof}

We now consider gradient estimate. Throughout the rest of this section, the covariant derivatives will be with respect to the metric $\tilde g$ on $\mathbf B^n$.

\begin{theo}\label{thmgradient}
  ({\bf Gradient estimate and exponential convergence.})  Let $\rho(\cdot, t)$ be a solution to the flow (\ref{eflow0}) on $[0,T]$. If $ (\phi')^2-\phi''\phi\ge 0$, then
\begin{equation}\label{}
    \di\max_{M(t)}e^{\alpha t}|\tilde\nabla \rho|^2\le \di\max_{M(0)}|\tilde\nabla \rho|^2,
\end{equation}
for some $\alpha>0$ which is independent of $t$.
\end{theo}
\begin{proof}
 Recall the evolution of $\P_t\rho$,
\begin{equation}
  \P_t \rho=n\frac{\phi'}{\phi}\o-\frac{1}{\phi\o^3}\tilde H,
  \label{}
\end{equation}
where $\tilde H$ was defined as in (\ref{tensors:rho 2}). We derive the evolution of $\frac{|\tilde\nabla \rho|^2}{2}$ below. Throughout the proof, we will work at a maximum point of the test funciton $\frac{|\tilde\nabla \rho|^2}{2}$, so that the following critical point conditions will hold,
\begin{equation}
  \tilde\n\o^2=\tilde\n\phi^2, \mathrm{or\quad } \tilde\n\o=\frac{\phi\phi'}{\o}\tilde\n\rho.
  \label{}
\end{equation}

First we have, at critical points of the test function,

\[
  \begin{array}[]{rll}
    \tilde\nabla \rho\tilde\nabla \tilde H =& \tilde\n \rho \tilde\n \Big[-\phi\o^2\tilde\Delta\rho+\phi\rho^i(\frac{|\tilde\n\rho|^2}{2})_i+\phi'\phi^2|\tilde\n\rho|^2+n\phi'\phi^2\o^2\Big]\\
\\
=& -\phi\o^2\tilde\n\rho\tilde\n\tilde\Delta\rho +\phi\rho_i\rho^k(\frac{|\tilde\n\rho|^2}{2})_{ik}\\
& -\rho^k(\phi\o^2)_k\tilde\Delta\rho+\rho^k(\phi'\phi^2)_k|\tilde\n\rho|^2+n\rho^k(\phi'\phi^2)_k\o^2+n\phi'\phi^2\rho^k(\phi^2)_k\\
  \end{array}
  \]

Note that
\[
  \begin{array}[]{rll}
      -\phi\o^2\tilde\n\rho\tilde\n\tilde\Delta\rho =& -\phi\o^2\rho^k(\tilde\n_i\rho_{ik}-\tilde R_{ik}\rho_i)\\
    =&-\phi\o^2\tilde\Delta \frac{|\tilde\n\rho|^2}{2} + \phi\o^2|\rho_{ij}|^2+\phi\o^2\tilde Ric(\tilde \n\rho,\tilde \n\rho).
  \end{array}
  \]

Thus
\begin{equation}
  \begin{array}[]{rll}
    \tilde\nabla \rho\tilde\nabla \tilde H =& -\phi(\o^2 \tilde g^{ik}-\rho^i\rho^k)(\frac{|\tilde\n\rho|^2}{2})_{ik}+ \phi\o^2|\rho_{ij}|^2+\phi\o^2\tilde Ric(\tilde \n\rho,\tilde \n\rho)\\
& -\rho^k(\phi\o^2)_k\tilde\Delta\rho+\rho^k(\phi'\phi^2)_k(|\tilde\n\rho|^2+n\o^2)+n\phi'\phi^2\rho^k(\phi^2)_k\\

  \end{array}
  \label{}
\end{equation}

Now we have,

\begin{equation}
   \begin{array}[]{rll}
     \P_t\frac{|\tilde\nabla \rho|^2}{2} = & \tilde\nabla \rho\tilde\nabla\rho_t \\
     \\
     =& n(\frac{\phi'}{\phi})'|\tilde\n\rho|^2\o+n(\frac{\phi'}{\phi})\tilde\n\rho\tilde\n\o-\tilde H\tilde\n\frac{1}{\phi\o^3}\tilde\n\rho  -\frac{1}{\phi\o^3}\tilde\n\rho\tilde\n \tilde H\\
  \\
  =&\frac{1}{\o^3}(\o^2 \tilde g^{ik}-\rho^i\rho^k)(\frac{|\tilde\n\rho|^2}{2})_{ik}-\frac{1}{\o}|\rho_{ij}|^2-\frac{1}{\o}\tilde Ric(\tilde \n\rho,\tilde \n\rho)\\
  & +\frac{1}{\phi\o^3}\rho^k(\phi\o^2)_k\tilde\Delta\rho-\frac{1}{\phi\o^3}\rho^k(\phi'\phi^2)_k(|\tilde\n\rho|^2+n\o^2)
  -\frac{n}{\phi\o^3}\phi'\phi^2\rho^k(\phi^2)_k\\
&+n(\frac{\phi'}{\phi})'|\tilde\n\rho|^2\o+n(\frac{\phi'}{\phi})\tilde\n\rho\tilde\n\o-\tilde H\tilde\n\frac{1}{\phi\o^3}\tilde\n\rho  \\
  \\

   \end{array}
  \label{}
\end{equation}

  Let $\mathcal L(\psi):= \P_t \psi -\frac{1}{\o^3}(\o^2\tilde g_{ij}-\rho^i\rho^j) \psi_{ij}$ be a parabolic operator for any function $\psi$ defined on $\mathbf B^n$.  Then applying the critical point conditions,

\begin{equation}
   \begin{array}[]{rll}
    \mathcal L(\frac{|\tilde\nabla \rho|^2}{2}) = &-\frac{1}{\o}|\rho_{ij}|^2-\frac{1}{\o}\tilde Ric(\tilde \n\rho,\tilde \n\rho)\\
  & +\frac{1}{\phi\o^3}\rho^k(\phi\o^2)_k\tilde\Delta\rho-\frac{1}{\phi\o^3}\rho^k(\phi'\phi^2)_k(|\tilde\n\rho|^2+n\o^2)-\frac{n}{\phi\o^3}\phi'\phi^2\rho^k(\phi^2)_k\\
&+n(\frac{\phi'}{\phi})'|\tilde\n\rho|^2\o+n(\frac{\phi'}{\phi})\tilde\n\rho\tilde\n\o-\tilde H\tilde\n\frac{1}{\phi\o^3}\tilde\n\rho  \\
  \\
= &-\frac{1}{\o}|\rho_{ij}|^2-\frac{1}{\o}\tilde Ric(\tilde \n\rho,\tilde \n\rho)+n(\frac{\phi'}{\phi})'|\tilde\n\rho|^2\o\\
  & -\frac{(\phi'\phi^2)'}{\phi\o^3}|\tilde\n\rho|^2(|\tilde\n\rho|^2+n\o^2)-\frac{n}{\phi\o^3}\phi'\phi^2(\phi^2)'|\tilde\n\rho|^2+n\frac{(\phi')^2}{\o}|\tilde\n\rho|^2\\
  &-\tilde H\tilde\n\frac{1}{\phi\o^3}\tilde\n\rho  +\frac{1}{\phi\o^3}\rho^k(\phi\o^2)_k\tilde\Delta\rho\\
  \\
= &-\frac{1}{\o}|\rho_{ij}|^2-\frac{1}{\o}\tilde Ric(\tilde \n\rho,\tilde \n\rho)-n(\frac{\phi'}{\phi})'|\tilde\n\rho|^2\o\\
  &  -\frac{(\phi'\phi^2)'}{\phi\o^3}|\tilde\n\rho|^2(|\tilde\n\rho|^2+n\o^2)-\frac{n}{\phi\o^3}\phi'\phi^2(\phi^2)'|\tilde\n\rho|^2+n\frac{(\phi')^2}{\o}|\tilde\n\rho|^2\\
  &+\tilde H (\frac{\phi'}{\phi^2\o^3}+\frac{3\phi'}{\o^5})|\tilde\n\rho|^2  +(\frac{\phi'}{\phi\o}+\frac{2\phi\phi'}{\o^3})|\tilde\n\rho|^2\tilde\Delta\rho\\

   \end{array}
  \label{}
\end{equation}

 At critical points,
 \begin{equation}
    \begin{array}[]{rll}
\tilde H = -\phi\o^2\tilde\Delta\rho+\phi'\phi^2|\tilde\n\rho|^2+n\phi'\phi^2\o^2
    \end{array}
    \label{}
  \end{equation}
and
  \begin{equation}
    \begin{array}[]{rll}
    \mathcal L(\frac{|\tilde\nabla \rho|^2}{2}) = &-\frac{1}{\o}|\rho_{ij}|^2-\frac{\phi\phi'}{\o^3}\tilde\Delta\rho|\tilde\n\rho|^2-\frac{1}{\o}\tilde Ric(\tilde \n\rho,\tilde \n\rho)+n(\frac{\phi'}{\phi})'|\tilde\n\rho|^2\o\\
&  -\frac{(\phi'\phi^2)'}{\phi\o^3}(|\tilde\n\rho|^2+n\o^2)|\tilde\n\rho|^2-\frac{n}{\o^3}\phi'\phi(\phi^2)'|\tilde\n\rho|^2+n\frac{(\phi')^2}{\o}|\tilde\n\rho|^2\\
  &+(\phi')^2(\frac{1}{\o^3}+\frac{3\phi^2}{\o^5})(|\tilde\n\rho|^2+n\o^2) |\tilde\n\rho|^2  \\
\\
= &-\frac{1}{\o}|\rho_{ij}|^2-\frac{\phi\phi'}{\o^3}\tilde\Delta\rho|\tilde\n\rho|^2-\frac{1}{\o}\tilde Ric(\tilde \n\rho,\tilde \n\rho)+n(\frac{\phi'}{\phi})'|\tilde\n\rho|^2\o\\
&-\frac{n}{\o^3}\phi'\phi(\phi^2)'|\tilde\n\rho|^2+n\frac{(\phi')^2}{\o}|\tilde\n\rho|^2+(\frac{3\phi'^2\phi^2}{\o^5}-\frac{\phi''\phi+\phi'^2}{\o^3})(|\tilde\n\rho|^2+n\o^2) |\tilde\n\rho|^2 . \\
\end{array}
    \label{}
  \end{equation}

  Recall the critical point conditions, by rotating the coordinates, we can pick $\rho_1=|\tilde\n\rho|$, thus
  \begin{equation}
    \rho_{11}=0, \mathrm{and}\quad \rho_{1j}=0, \forall j=2,\cdots,n.
    \label{}
  \end{equation}
  Moreover, we can diagonalize $\rho_{jk}$ for $j,k=2,\cdots,n$ at the crical point and $\tilde\Delta \rho:=\sum_{j\ge2}\rho_{jj}$. By completing the square, we have

  \begin{equation}
    \begin{array}[]{rll}
    \mathcal L(\frac{|\tilde\nabla \rho|^2}{2})
    = &-\frac{1}{\o}\sum_{j\ge2}\Big(\rho_{jj}+\frac{1}{2}\frac{\phi\phi'}{\o^2}|\tilde\n\rho|^2\Big)^2+\frac{n-1}{4}\frac{\phi^2\phi'^2}{\o^5}|\tilde\n\rho|^4-\frac{1}{\o}\tilde Ric(\tilde \n\rho,\tilde \n\rho)+n(\frac{\phi'}{\phi})'|\tilde\n\rho|^2\o\\
&-\frac{n}{\o^3}\phi'\phi(\phi^2)'|\tilde\n\rho|^2+n\frac{(\phi')^2}{\o}|\tilde\n\rho|^2+(\frac{3\phi'^2\phi^2}{\o^5}-\frac{\phi''\phi+\phi'^2}{\o^3})(|\tilde\n\rho|^2+n\o^2) |\tilde\n\rho|^2  \\
\\

   = &-\frac{1}{\o}\sum_{j\ge2}\Big(\rho_{jj}+\frac{1}{2}\frac{\phi\phi'}{\o^2}|\tilde\n\rho|^2\Big)^2\\
& +\frac{n-1}{4\o^5}(\phi\phi')^2|\tilde\n\rho|^4-\frac{1}{\o}\tilde Ric(\tilde \n\rho,\tilde \n\rho)+n(\frac{\phi'}{\phi})'|\tilde\n\rho|^2\o\\
&-\frac{2n}{\o^3}(\phi')^2\phi^2|\tilde\n\rho|^2+n\frac{(\phi')^2}{\o}|\tilde\n\rho|^2+(\frac{3\phi'^2\phi^2}{\o^5}-\frac{\phi''\phi+\phi'^2}{\o^3})(|\tilde\n\rho|^2+n\o^2) |\tilde\n\rho|^2  \\
%\\
%&+(\frac{3\phi'^2\phi^2}{\o^5}-\frac{\phi''\phi+\phi'^2}{\o^3})(|\tilde\n\rho|^2+n\o^2) |\tilde\n\rho|^2  \\
\\
= &-\frac{1}{\o}\sum_{j\ge2}\Big(\rho_{jj}+\frac{1}{2}\frac{\phi\phi'}{\o^2}|\tilde\n\rho|^2\Big)^2-\frac{1}{\o}\tilde Ric(\tilde \n\rho,\tilde \n\rho)\\
&+ \frac{|\tilde\n\rho|^2}{\o^5}\Big[n(\frac{\phi'}{\phi})'\o^6+n\phi'^2\o^4+ \frac{n-1}{4}\phi^2\phi'^2|\tilde\n\rho|^2-2n\phi'^2\phi^2\o^2\\
&+(2\phi'^2\phi^2-\phi'^2|\tilde\n\rho|^2-\phi''\phi\o^2)(|\tilde\n\rho|^2+n\o^2)  \Big]

\end{array}
    \label{}
  \end{equation}
  Notice that

  \begin{equation}
    \begin{array}[]{rll}
&n(\frac{\phi'}{\phi})'\o^6+(2\phi'^2\phi^2-\phi'^2|\tilde\n\rho|^2-\phi''\phi\o^2)(|\tilde\n\rho|^2+n\o^2)\\
=&n(\frac{\phi'}{\phi})'\o^6-(\phi''\phi-\phi'^2)(|\tilde\n\rho|^2+n\o^2)\o^2+(\phi'^2\phi^2-2\phi'^2|\tilde\n\rho|^2)(|\tilde\n\rho|^2+n\o^2)\\
=&n\frac{\phi''\phi-\phi'^2}{\phi^2}|\tilde\n\rho|^4\o^2+(n-1)(\phi''\phi-\phi'^2)|\tilde\n\rho|^2\o^2+(\phi'^2\phi^2-2\phi'^2|\tilde\n\rho|^2)(|\tilde\n\rho|^2+n\o^2)\\

\end{array}
    \label{}
  \end{equation}
Thus

  \begin{equation}
    \begin{array}[]{rll}
    \mathcal L(\frac{|\tilde\nabla \rho|^2}{2})
    = &-\frac{1}{\o}\sum_{j\ge2}\Big(\rho_{jj}+\frac{1}{2}\frac{\phi\phi'}{\o^2}|\tilde\n\rho|^2\Big)^2-\frac{1}{\o}\tilde Ric(\tilde \n\rho,\tilde \n\rho)\\
&+ \frac{|\tilde\n\rho|^2}{\o^5}\Big[n\frac{\phi''\phi-\phi'^2}{\phi^2}|\tilde\n\rho|^4\o^2+(n-1)(\phi''\phi-\phi'^2)|\tilde\n\rho|^2\o^2\\
&
  -(n+2)\phi'^2|\tilde\n\rho|^4-\frac{3}{4}(n-1)\phi^2\phi'^2|\tilde\n\rho|^2 \Big]

\end{array}
    \label{gradient evolution}
  \end{equation}

By the assumption $\tilde Ric\ge (n-1)K\tilde g$ with $K>0$. As far as $\phi'^2-\phi''\phi\ge0$, we have
\begin{align*}
    \mathcal L(\frac{|\tilde\nabla \rho|^2}{2})\le 0.
\end{align*}
By the maximum principle, there is a uniform upper bound for $|\tilde\n\rho|$.  Moreover, with the uniform $C^0$ and gradient estimates, we now have

  \begin{equation}
    \begin{array}[]{rll}
      \mathcal L(\frac{|\tilde\nabla \rho|^2}{2})\le -\frac{1}{\o}\tilde Ric(\tilde \n\rho,\tilde \n\rho)\le -\alpha |\tilde\n\rho|^2,
\end{array}
    \label{}
  \end{equation}
where $\alpha>0$ is a uniform constant depending on the upper bound of $|\tilde\n\rho|^2$. This implies the exponential convergence for the case $K>0$. The exponential convergence also holds for the case $K=0$, which will be dealt with separately in the last section.
\end{proof}

\section{Evolution of support function  and the mean curvature}\label{uH}

In this section, we prove a uniform upper bound estimate for the mean curvature $H$ along the flow. We also prove a uniform positive lower bound for support function $u$ under condition $(\phi^{'})^2-\phi\phi^{"}>0$. Although the results in this section are not needed for proving the main theorem, as important properties of the flow itself, we include them here for completeness and future interests.

Suppose that the metric on $(N^{n+1}, \bar{g})$ is a warped product of the form \eqref{warped metric}. We denote the Riemannian metric and the Levi-Civita connection of $(N^{n+1}, \bar{g})$ by  $\langle \cdot, \cdot\rangle$ and $\bar{\nabla}$, respectively. The conformal Killing field is $X=\bar{\nabla} \Phi$ (recall $\Phi'(r)=\phi(r)$) and
$X$ satisfies
\begin{equation}\label{eq_X} \langle \bar \nabla_Y X, Z\rangle=\bar{\nabla}_Y\bar{\nabla}_Z\Phi =\phi'\langle Y, Z\rangle, \end{equation}  where $\bar{\nabla}\bar{\nabla}\Phi$ is the Hessian of $\Phi$ with respect to $\bar{g}$, and $Y$ and $Z$ are any two vector fields on $N^{n+1}$.

Let $M_t$ by a family of hypersurfaces evolves by  \eqref{eflow0}:

\[ \frac{\partial F}{\partial t}= f\nu,\] where $f$ is given by

\begin{equation}\label{eq_f} f=n \phi'-u H.\end{equation}  The outward unit normal $\nu$ of $M_t$ evolves by
\begin{equation}\label{eq_nu}\frac{\partial \nu}{\partial t}=-\nabla^{M_t} f, \end{equation} where we use $\nabla^{M_t}$ and $\Delta^{M_t}$ to denote the gradient and Laplace operators on $M_t$, with respect to the induced metric.

We first compute the evolution equation of $u$. Note that in view of the evolution equation the relevant parabolic operator for any geometric quantity defined on $M_t$ is $\partial_t - u \Delta^{M_t}$.

We compute using \eqref{eq_X} and \eqref{eq_nu}: \begin{equation}\begin{split}\label{dt_support}\partial_t\langle X, \nu\rangle&= f\langle \bar{\nabla}_\nu X, \nu\rangle+\langle X, \nabla^{M_t}(u H)\rangle-n \langle X, \nabla^{M_t} \phi'|_{M_t}\rangle\\
&=f\phi'|_{M_t}+\langle X, \nabla^{M_t}(u H)\rangle-n \langle X, \nabla^{M_t} \phi'|_{M_t}\rangle.\end{split}\end{equation}

Choosing the orthonormal frame $\{e_i\}_{i=1\cdots n}$ to $M_t$ such that $\nabla^{M_t}_{e_i} e_j=0$ at a point where the following calculation is conducted:
\begin{equation}\label{lap_u1}\Delta^{M_t} u =e_i\big\langle \bar{\nabla}_{e_i} X , \nu\rangle+e_i \langle X, \bar{\nabla}_{e_i} \nu\rangle.\end{equation}

The first term vanishes by \eqref{eq_X}. Recall that $\bar{\nabla}_{e_i} \nu =h_{ij} e_j$ ($\nabla^{M_t}_{e_i} e_j=0$), where $h_{ij}$ is the second fundamental form of $M_t$ and the second term is equal to
\[\begin{split}e_i( h_{ij}\langle X, e_j\rangle)&= e_i(h_{ij})\langle X, e_j\rangle+h_{ij}\langle \nabla_{e_i} X, e_j\rangle+h_{ij}\langle X, \nabla_{e_i} e_j\rangle\\
&=(\nabla_i^{M_t} h_{ij})\langle X, e_j\rangle+\phi' H-\sum h_{ij}^2\langle X, \nu\rangle. \end{split} \]

Plugging this back to \eqref{lap_u1} and multiplying each term by $u=\langle X, \nu\rangle$, we obtain

\begin{equation}\label{elliptic_support} \begin{split}& u \Delta^{M_t} u= u (\nabla_i^{M_t} h_{ij})\langle X, e_j\rangle+u\phi' H-u^2\sum h_{ij}^2. \end{split} \end{equation}

Combining \eqref{dt_support} and \eqref{elliptic_support}, we obtain (we use $\phi'$ to denote $\phi'|_{M_t}$ in the following)

\begin{equation}\label{parabolic_u0}\begin{split} \partial_t u-  u \Delta^{M_t} u& =   u^2\sum h_{ij}^2 -2 \phi' H u +n(\phi')^2+H \langle X, \nabla^{M_t}u \rangle          \\
&+ u  \big[ \langle X,  \nabla^{M_t} H\rangle- (\nabla_i^{M_t} h_{ij})\langle X, e_j\rangle \big] -  n \langle X, \nabla^{M_t} \phi'\rangle \\
&= (\sum h_{ij}^2-\frac{H^2}{n})u^2 +\frac{1}{n}(H u -n \phi' )^2 + H \langle K, \nabla^{M_t} u\rangle \\
&+ u   \big[ \langle X,  \nabla^{M_t} H\rangle- (\nabla_i^{M_t} h_{ij})\langle X, e_j\rangle \big] -  n \langle X, \nabla^{M_t} \phi'\rangle. \\
\end{split}\end{equation}

Note that $ \langle X,  \nabla^{M_t} H\rangle- (\nabla_i^{M_t} h_{ij})\langle X, e_j\rangle $ can be expressed in terms of  $\bar{Ric}(X^\top, \nu)$ where $X^\top=\langle X, e_i\rangle e_i$ is the component of $X$ that is tangential to $M_t$. This is the same as the term that appears in the monotonicity formula.
However, in this case, if we only want to prove that $u>0$ is preserved along the flow, the sign of the term $\bar{Ric}(X^\top, \nu)$ does not matter.

On the other hand, we compute \[\nabla^{M_t}\phi'=\bar{\nabla}\phi'-\langle \bar{\nabla}\phi', \nu\rangle \nu\] and $\bar{\nabla}\phi'=\phi''\bar{\nabla} r=\phi''\partial_r=\frac{\phi''}{\phi} X$. Therefore, \[\langle X, \nabla^{M_t}\phi'\rangle=\frac{\phi''}{\phi}\langle X, X-u \nu\rangle=\frac{\phi''}{\phi} (\phi^2-u^2).\]

Plugging this into \eqref{parabolic_u0}, we obtain

\begin{equation}\label{parabolic_u} \begin{split} \partial_t u-  u \Delta^{M_t} u&= (\sum h_{ij}^2-\frac{H^2}{n})u^2 +\frac{1}{n}(H u -n \phi' )^2 + H \langle K, \nabla^{M_t} u\rangle \\
&+ u   \big[ \langle X,  \nabla^{M_t} H\rangle- (\nabla_i^{M_t} h_{ij})\langle X, e_j\rangle \big] -  n\frac{\phi''}{\phi} (\phi^2-u^2)  \\
\end{split}. \end{equation}

Recall in (\ref{Ricci iden}) and (\ref{Ric normal}), we have derived $R_{\nu e_i}$ and
\begin{align*}
    &\big[ \langle X,  \nabla^{M_t} H\rangle- (\nabla_i^{M_t} h_{ij})\langle X, e_j\rangle \big] =- g^{ij}\bar \nabla_i\Phi \bar R_{j\nu}\\
    =&(n-1)(K-\phi'^2+\phi\phi'')\frac{|\tilde\nabla\rho|^2}{\omega^3}+(\tilde Ric_{ik}-(n-1)K\tilde g_{ik})\frac{\rho_k\rho_i}{\omega^3}\ge 0
\end{align*}
Thus (\ref{parabolic_u}) can be simplified as
\begin{align}
     \partial_t u-  u \Delta^{M_t} u&\ge H \langle K, \nabla^{M_t} u\rangle +  n(\phi'^2-\phi''\phi) -2\phi'H u.
    \label{parabolic_u inequ}
\end{align}

\medskip

We now switch to the evolution of mean curvature $H$ along the flow.
\begin{prop}
    \label{prop : evolution of H}
    Along the flow (\ref{eflow0}), the mean curvature of a graphical hypersurface evolves as the follows
    \begin{align}
        \begin{array}[]{rll}
            \P_t H =&u\Delta H+H\n H\n \Phi +2\n H\n u+\phi'(H^2-n|A|^2)\\
&-\frac{n}{\phi^2}\Big[(n-1)\phi'(K-\phi'^2)+(n-2)\phi\phi'\phi''+\phi^2\phi'''\Big](1-\frac{u^2}{\phi^2})\\
&-n\frac{\phi'}{\phi^6}(\tilde R_{ij}-(n-1)K\tilde g_{ij})\rho_i\rho_ju^2.
        \label{eH}
\end{array}
    \end{align}

\end{prop}

\begin{proof}
    Using Proposition \ref{prop2.1} and replacing the general $f$ by $n\phi'-Hu$, we have
    \[
        \begin{array}[]{rll}
            \P_t H=&-\Delta_gf-f|A|^2-f\bar R_{\nu\nu}\\
            =&-\Delta_g(n\phi'-Hu)-f|A|^2-f\bar R_{\nu\nu}\\
            =&u\Delta H+H\Delta u+2\n H\n u-n\Delta \phi'-f|A|^2-f\bar R_{\nu\nu}\\
            =&u\Delta H+H\Delta u+2\n H\n u-n\Delta \phi'-f|A|^2-f\bar R_{\nu\nu}\\
\end{array}
    \]
    Using Proposition \ref{prop hessian}, we have
    \[
        \begin{array}[]{rll}
            \Delta \phi' =&\frac{\phi''}{\phi} \Delta \Phi + \nabla \frac{\phi''}{\phi}\nabla \Phi\\
            =&\frac{\phi''}{\phi} f + \frac{1}{\phi}( \frac{\phi''}{\phi})'|\nabla \Phi|^2,
        \end{array}
    \]
   where $f=n\phi'-Hu$.

    Combining with Proposition \ref{prop hessian}, we have

    \[
        \begin{array}[]{rll}
            \P_t H =&u\Delta H+H\Big[\n H\n \Phi +\bar R_{\nu i}\n_i\Phi+H\phi'-|A|^2u\Big]\\
            &+2\n H\n u-n\frac{\phi''}{\phi} f -n\frac{1}{\phi}( \frac{\phi''}{\phi})'|\nabla \Phi|^2-f|A|^2-f\bar R_{\nu\nu}\\
            \\
            =&u\Delta H+H\n H\n \Phi +2\n H\n u+H\bar R_{\nu i}\n_i\Phi+\phi'(H^2-n|A|^2)\\
            &-(n\frac{\phi''}{\phi} +\bar R_{\nu\nu})f -n\frac{1}{\phi}( \frac{\phi''}{\phi})'|\nabla \Phi|^2.
\end{array}
    \]

   Replacing $f$ by $n\phi'-Hu$, we get
   \begin{align}
         \begin{array}[]{rll}
            \P_t H =&u\Delta H+H\n H\n \Phi +2\n H\n u+\phi'(H^2-n|A|^2)\\
            &+H\Big[\bar R_{\nu i}\n_i\Phi+u(n\frac{\phi''}{\phi} +\bar R_{\nu\nu})\Big] \\
            &-n\Big[(n\frac{\phi''}{\phi} +\bar R_{\nu\nu})\phi' +\frac{1}{\phi}( \frac{\phi''}{\phi})'|\nabla \Phi|^2\Big].
        \end{array}
       \label{H evolution equ 1}
   \end{align}
   By (\ref{Ric normal}) and the definition of $u=\frac{\phi^2}{\o}$, we find out the mean curvature term in the second line of (\ref{H evolution equ 1}) vanishes, i.e.,
   \begin{align}
\bar R_{\nu i}\n_i\Phi+u(n\frac{\phi''}{\phi} +\bar R_{\nu\nu})=0.
\label{H term vanish}
   \end{align}
   On the other hand, since $|\n \Phi|^2= \frac{\phi^2|\tilde\n\rho|^2}{\o^2}$, using (\ref{Ric normal}) again, we can compute the last line in (\ref{H evolution equ 1}),

   \begin{align}
      & \Big[(n\frac{\phi''}{\phi} +\bar R_{\nu\nu})\phi' +\frac{1}{\phi}( \frac{\phi''}{\phi})'|\nabla \Phi|^2\Big]\nonumber\\
      =&\Big[(n-1)\frac{K-\phi'^2+\phi\phi''}{\phi^2} \phi' +\phi( \frac{\phi''}{\phi})'\Big]\frac{|\tilde\nabla \rho|^2}{\o^2} +(\tilde R_{ij}-(n-1)K\tilde g_{ij})\frac{\rho_i\rho_j}{\phi^2\o^2}\phi'.
       \label{H evolution equ 2}
   \end{align}

 This yields that

 \begin{align}
         \begin{array}[]{rll}
            \P_t H =&u\Delta H+H\n H\n \Phi +2\n H\n u+\phi'(H^2-n|A|^2)\\
&-n\Big[(n-1)\frac{K-\phi'^2+\phi\phi''}{\phi^2} \phi' +\phi( \frac{\phi''}{\phi})'\Big]\frac{|\tilde\nabla \rho|^2}{\o^2} -n(\tilde R_{ij}-(n-1)K\tilde g_{ij})\frac{\rho_i\rho_j}{\phi^2\o^2}\phi'.
\end{array}
     \label{H evolution equ 3}
 \end{align}
 Plugging the definition of $u$ into (\ref{H evolution equ 3}), we finished the proof.
\end{proof}

\begin{rema}
    By direct computations, if the ambient space is substatic, namely satisfying conditions H1-H4 in Brendle's work \cite{Br}, the last two lines in (\ref{eH}) are nonnegative, which immediately implies a  uniform upper bound for the mean curvature by the maximum principle.
\end{rema}
Next, we show that the mean curvature $H$ has a uniform upper bound in general without assuming the substatic conditions for the ambient metric.
\begin{theo}\label{H upper bound} Suppose
     \begin{equation}\label{phi-inc} \phi^{'}(r)>0, \quad \forall r\in (r_0, \bar {r}).\end{equation}   The mean curvature of the graphical hypersurface evolving along the flow (\ref{eflow0}) has a uniform upper bound,
    \[
        \max_{t>0}H(\cdot, t)\le C,
    \]
    where $C$ is a uniform constant which is independent of $t$.

\end{theo}

\begin{proof}
    From the evolution equations of $H$ and $\Phi$, we obtain
    \begin{equation}
        \begin{array}[]{rll}
            \mathcal L(H+\Phi) \le & H\n (H+\Phi)\n \Phi +2\n (H+\Phi)\n u+\phi'(H^2-n|A|^2)\\
            &-H|\nabla \Phi|^2-2\nabla \Phi\n u\\
&-\frac{n}{\phi^2}\Big[(n-1)\phi'(K-\phi'^2)+(n-2)\phi\phi'\phi''+\phi^2\phi'''\Big](1-\frac{u^2}{\phi^2})\\
\\
= & H\n (H+\Phi)\n \Phi +2\n (H+\Phi)\n u+\phi'(H^2-n|A|^2)\\
            &-H|\nabla \Phi|^2-2h(\nabla \Phi,\nabla \Phi)\\
&-\frac{n}{\phi^2}\Big[(n-1)\phi'(K-\phi'^2)+(n-2)\phi\phi'\phi''+\phi^2\phi'''\Big](1-\frac{u^2}{\phi^2})
        \end{array}
        \label{equ 1 : proof theo}
    \end{equation}
    Since $|\n \Phi|^2=\phi^2 (1-\frac{u^2}{\phi^2})$ and functions related to $\phi$ are all uniformly bounded from above and below, we conclude that the term in the last line of euqation (\ref{equ 1 : proof theo}) is uniformly bounded by $C_1|\n\Phi|^2$ where $C_1$ is a uniform constant which does not depend on $t$. Thus
    \begin{equation}
        \begin{array}[]{rll}
            \mathcal L(H+\Phi) \le & H\n (H+\Phi)\n \Phi +2\n (H+\Phi)\n u+\phi'(H^2-n|A|^2)\\
            &-H|\nabla \Phi|^2-2h(\nabla \Phi,\nabla \Phi)+C_1|\n\Phi|^2
        \end{array}
        \label{equ 2 : proof theo}
    \end{equation}

    At a maximum point of the test function $H+\Phi$, we can choose normal coordinates, so that the metric tensor at the point is the identity matrix. If we choose a coordinate system so that the $x_1$ axis direction is the direction of $\n \Phi$, then $\Phi_i=0$, for all $i=2,\cdots,n$.

\begin{equation}
        \begin{array}[]{rll}
            \mathcal L(H+\Phi) \le & -\phi'(n|A|^2-H^2)-(H+2h_{11}-C_1)|\n\Phi|^2
        \end{array}
        \label{equ 3 : proof theo}
    \end{equation}
where we have used the critical point condition to eliminate the gradient terms.

  Without loss of generality, we assume $\frac{H}{2}-C_1\ge 0$, otherwise the test function $H+\Phi$ is uniformly bounded from above by a constant and the proof is done. This reduces (\ref{equ 3 : proof theo}) to

\begin{equation}
        \begin{array}[]{rll}
            \mathcal L(H+\Phi) \le & -\phi'(n|A|^2-H^2)-(\frac{H}{2}+2h_{11})|\n\Phi|^2
        \end{array}
        \label{equ 3.1 : proof theo}
    \end{equation}

Next we consider two different cases.\\

{\bf Case I:} Suppose at the maximum point, $\frac{H}{2}+2h_{11}>0$, then (\ref{equ 3 : proof theo}) is reduced to
\[
        \begin{array}[]{rll}
            \mathcal L(H+\Phi) \le & 0,
        \end{array}
    \]
    and by the maximum principle $H + \Phi$ is bounded, and so is $H$.\\

    {\bf Case II:} Suppose that at the maximum point, $\frac{H}{2}+2h_{11}\le 0$. Let $\lambda_i:= h_{ii}$, denote $A_{*}=\{\lambda_2, \cdots, \lambda_n\}$, $H_{*}=\lambda_{2}+\cdots +\lambda_{n}$. We have $|A|^2 = \lambda_1^2+|A_{*}|^2$, and $H = \lambda_1+H_{*}$. Then
    \begin{equation}
        -\lambda_1\ge \frac{1}{4}H.
        \label{condition : theo}
    \end{equation}
and
\begin{equation}
        \begin{array}[]{rll}
            n|A|^2-H^2=& (n|A_{*}|^2-H_{*}^2)+(n+1)\lambda_1^2-2\lambda_1 H\\
            \ge &(n+1)\lambda_1^2-2\lambda_1 H\\
        \end{array}
        \label{equ 4 : proof theo}
    \end{equation}
    where we used Cauchy-Schwartz inequality in the inequality.

  Recall $ \phi'\ge C_2>0$ and $0\le |\n\Phi|^2\le C_3$. Applying (\ref{equ 4 : proof theo}), (\ref{equ 3.1 : proof theo}) yields
\[
        \begin{array}[]{rll}
            \mathcal L(H+\Phi) \le &-C_2\Big[(n+1)\lambda_1^2-2\lambda_1H\Big] -C_3(\frac{H}{2}+2\lambda_1)\\
            =&  -C_2(n+1)\lambda_1^2-C_3\frac{H}{2}-2(-\lambda_1)(C_2H -C_3)).
        \end{array}
    \]
    We assume that $H>\frac{C_3}{C_2}\ge 0$, otherwise $H$ has an upper bound. This yields
\[
        \begin{array}[]{rll}
            \mathcal L(H+\Phi) \le&0.
        \end{array}
    \]
    By the maximum principle, we conclude that the test function $H+\Phi$ has a uniform upper bound.
\end{proof}

\medskip

By Theorem \ref{H upper bound}, we know that the mean curvature $H$ has a uniform upper bound. If we further assume the strict inequality $\phi'^2-\phi''\phi>0$, then by the uniform $C^0$ estimate and uniform upper bound for $H$, the second term on the right hand side of (\ref{parabolic_u inequ}) is the dominant term as $u\rightarrow 0$. Thus by the maximum principle, we have shown the uniform positivity of the support function $u$ which in turn yields the needed gradient estimate.

\begin{prop}
    \label{gradient by u}
    Let $M_0\subset \mathbf N^{n+1}$ be a smooth graphical hypsurface with support function $u>0$. Assume condition (\ref{phi-inc} )  and
    \begin{align*}
        (\phi')^2-\phi''\phi>0.
    \end{align*}
    then there exists a uniform constant $C>0$ independent of $t$, such that
    \begin{align*}
        u(\cdot,t)\ge C>0
    \end{align*}
    as long as the solution of the flow \eqref{eflow0} exists.
\end{prop}

\section{Proofs of Theorem 1.1 and Theorem 1.2}\label{proofs}

Since \eqref{ivp} is a quasilinear parabolic equation, it follows from Proposition \ref{rho bounded} and Theorem \ref{thmgradient} that the flow is uniformly parabolic. The longtime existence and regularity follow from the standard parabolic theory. The solution converges exponentially to a slice $\rho=constant$ by Theorem \ref{thmgradient}. This proves Theorem \ref{main thm}. Theorem \ref{main thm2} follows from the following proposition.

\begin{prop}\label{p2}
    Let $\O\subset \mathbf N^{n+1}$ be a domain bounded by a smooth graphical hypersurface $M$ and $S(r_0)$.    We assume $\phi(r)$ and  $\tilde g$ satisfy the conditions \eqref{condition} in Theorem \ref{main thm}, then
    \begin{align}
        Area(M)\ge Area (S(r^{*})),
        \label{iso prob2}
    \end{align}
    where $r^{*}$ is the unique real number in $[r_0, \bar{r}]$ such that volume of $B(r^{*})$ enclosed by $S(r^{*})$ and $S(r_0)$ is equal to $Vol(\O)$. If equality in (\ref{iso prob2}) holds, then $M$ must be umblic. If, in addition to \eqref{condition}, $(\phi')^2-\phi''\phi< K \text{ on } [r_0, \bar{r}]$ then  ``=" is attained in \eqref{iso prob} if and only if $M$ is a level set of $r$.
\end{prop}

\begin{proof} By Theorem \ref{thm mono hyperbolic}, the area of evolving hypersurfaces along flow \ref{eflow0} is decreasing and the enclosed volume is preserved. By Theorem \ref{main thm}, the flow converges to a slice $S(r^{*})$ and we must have $Vol(B(r^{*}))=Vol(\O)$. Note that $Vol(B(r))$ is strictly increasing in $r$, thus $r^{*}$ is unique. This proves inequality (\ref{iso prob2}). If equality holds there, \eqref{mono 2 hyperbolic}
must be an equality. Therefore $\sigma_2=\frac{n-1}{2n}\sigma_1^2$ at every point, this implies that $M$ is umbilic. If either  $(\phi')^2-\phi''\phi< K$ or $\tilde{Ric}> (n-1)K \tilde{g}, \text{ on } [r_0, \bar{r}]$, (\ref{mono 2 hyperbolic}) and (\ref{Ric normal}) imply $\tilde \nabla \rho \equiv 0$. That is, $M$ is a slice. \end{proof}

\section{Conditions on the warping function $\phi$}\label{photon}

We illustrate that both the lower bound and upper bound in \eqref{condition} are necessary in certain sense and have geometric or physics interpretations.

The lower bound of $(\phi')^2-\phi''\phi$ is closely related to the notion of ``photon spheres". For each warped product space with a Riemannian metric $dr^2+\phi^2(r) \tilde{g}_{ij} du^i du^j$, there is an associated static spacetime $\mathfrak{S}$ with the spacetime metric \begin{equation}\label{spacetime}-(\phi'(r))^2 dt^2+dr^2+\phi^2(r) \tilde{g}_{ij} du^i du^j.\end{equation}
Such a spacetime has a natural conformal Killing-Yano two form and is of great interest in general relativity, see \cite[Remark 3.7]{WWZ}.
For the Schwarzschild manifold with the Riemannian metric $\frac{1}{1-\frac{m}{s}}ds^2+s^2 \tilde{g}_{ij} du^i du^j$ (after a change of coordinates $s=\phi(r)$), the associated static spacetime is the Schwarzschild spacetime with the spacetime metric
\[-(1-\frac{m}{s})dt^2+\frac{1}{1-\frac{m}{s}}ds^2+s^2 \tilde{g}_{ij} du^i du^j.\]

We recalled that a hypersurface is said to be \textit{totally umbilical}  if the second fundamental form is proportional to the induced metric (the first fundamental form).
A totally umbilical timelike hypersurface $\mathfrak{T}$ of a spacetime $\mathfrak{S}$ is called a \textit{photon sphere} \cite{CVE}, where null geodesics are trapped (i.e. a null geodesic which is initially tangent to $\mathfrak{T}$  remains within the hypersurface $\mathfrak{T}$). This is easily seen from the following relation: for any null vector field $X$,
\[\nabla^\mathfrak{S}_X X=\nabla^\mathfrak{T}_X X,\] where $\nabla^\mathfrak{S}$ and $\nabla^\mathfrak{T}$ are the covariant derivatives of $\mathfrak{S}$ and $\mathfrak{T}$, respectively.

We claim that the equation $(\phi')^2-\phi \phi''=0$ characterizes exactly the location of the photon sphere.

\begin{prop} For a spacetime $\mathfrak{S}$ with metric of the form \eqref{spacetime}, $(\phi')^2-\phi \phi''=0$ at $r_0$ if and only if $r=r_0$ is a photon sphere.

\end{prop}

\begin{proof} It suffices to prove that the hypersurface $\mathfrak{T}$ defined by $r=r_0$ is totally umbilical. The induced metric of $\mathfrak{T}$ is
\begin{equation}\label{induced_metric}-(\phi')^2 dt^2+\phi^2 \tilde{g}_{ij} du^i du^j\end{equation} and the unit outward normal of $\mathfrak{T}$ is $\frac{\partial}{\partial r}$.  We compute
\[\langle\nabla^\mathfrak{S}_{\partial_ t} \partial_r, \partial_t\rangle=\frac{1}{2} \partial_r (-(\phi'))^2=-\phi'\phi''.\]
and  \[\langle\nabla^\mathfrak{S}_{\partial_i} \partial_r,  \partial_j\rangle=\frac{1}{2} \partial_r (\phi^2 \tilde{g}_{ij})=\phi \phi'\tilde{g}_{ij},\] therefore the second fundamental form of $\mathfrak{T}$ is
\[ -\phi'\phi'' dt^2+\phi \phi'\tilde{g}_{ij} du^i du^j.\] Comparing this with \eqref{induced_metric} yields the desired conclusion.
\end{proof}

\subsection{The stability condition}\label{stability}

We recall the following notion of ``stable constant mean curvature" (stable CMC).
A hypersurface $M^n$ in $(N^{n+1}, \bar{g})$ is  a {\it stable CMC} if (1) $M^n$ is of constant mean curvature and (2) the following expression is non-negative for any smooth function $u$ on $M$ with $\int_M u=0$,
\begin{equation}\label{2nd_var}\int_M \left[ |\nabla u|^2- (|h|^2+\bar{Ric}(\nu, \nu)) u^2\right], \end{equation}
where $|h|^2$ is the norm square of the second fundamental form of $M$, $\bar{Ric}$ is the Ricci curvature of $(N^{n+1}, \bar{g})$, and $\nu$ is
the outward unit normal of $M$. The formula is  the same as the second variation for minimal submanifolds (see for example equation (1) on page 11 of \cite{CY}) except one requires the additional condition $\int_M u=0$, which corresponds to the volume constraint.

Suppose $(N^{n+1}, \bar{g})$ is a warped product and $\bar{g}=dr^2+\phi^2(r)\tilde{g}$ where $\tilde{g}$ is a Riemannian metric on a closed  manifold $B$. We recall that the first eigenvalue $\lambda_1(\tilde{g})$ of $\tilde{g}$ is defined to be

\[ \lambda_1(\tilde{g})=\inf_{\int_B u=0} \frac{\int_B |\nabla u|^2}{\int_B u^2}.\]

\begin{prop} Suppose $(N^{n+1}, \bar{g})$ is a warped product with $\bar{g}=dr^2+\phi^2(r)\tilde{g}$ and $M$ is a level set of $r$, say $M=\{ r=r_1\}$. If
\begin{equation}\label{stable_condition}
\begin{aligned}
       &\lambda_1(\tilde{g})\geq n( (\phi')^2-\phi \phi'') \text{ at } r=r_1,
    \end{aligned}
    \end{equation}
then $M$ is a stable CMC.
\end{prop}

\begin{proof}

The induced metric on $M$ is
$g_{ij}=\phi^2 (r_1) \tilde{g}_{ij}$ and the second fundamental form of $M$ is $h_{ij}=\phi(r_1)\phi'(r_1) \tilde{g}_{ij}$. Therefore,

\[|h|^2=g^{ij} g^{kl} h_{ik} h_{jl}= n\phi^{-2}(\phi')^2.\]

The outward unit normal of $M$ is $\nu=\partial_r $ and by Lemma \ref{Ricci lemma}, we have
\[\bar{Ric}(\nu, \nu)=-n\frac{\phi''}{\phi}.\]

Note that the first eigenvalue of $M$ is $\frac{\lambda_1(\tilde{g})}{\phi^2}$.  The expression \eqref{2nd_var} can be estimated by

\[\int_M \left[ |\nabla u|^2- u^2(|h|^2+\bar{Ric}(\nu, \nu))\right] \geq  [\lambda_1(\tilde{g}) -n((\phi')^2-\phi \phi'')] \int_M \frac{u^2}{\phi^2}, \] which is non-negative by the assumption.

\end{proof}

\begin{rema}

The condition \eqref{condition} is also necessary. Were it violated, one can take a first eigenfunction $u$ of $M$ with $\int_M u=0$ and deform the level set $M$ with the speed function $u$. The will have the effect of fixing the volume constraint and decreasing the surface area. We refer to \cite{LW} for a concrete example and also \cite{E, Ri} for the stability condition.
\end{rema}

That the condition  $ K \geq  (\phi')^2-\phi \phi''$ implies the stability follows from a theorem of Lichnerowicz \cite{L}:
\begin{theo}
Let $(B, \tilde{g}) $ be an n-dimensional  closed Riemannian manifold with positive Ricci curvature
$\tilde{Ric} \geq (n-1)K\tilde{g}$ for a constant  $K > 0$. The first eigenvalue of $\tilde{g}$ satisfies \[\lambda_1(\tilde{g}) \geq nK.\]
\end{theo}

\section{The convergence of the flow}

To end the paper, we discuss the convergence of flow (\ref{eflow0}) when $K=0$ in (\ref{condition}). For example, then base manfold $\mathbf B^n$ is Ricci flat.  As we assume $0\le (\phi^{'})^2-\phi\phi^{''}\le K$, this forces $(\phi^{'}(r))^2=\phi(r)\phi^{''}(r), \forall r\in [r_0,\bar r]$.  Note that Proposition \ref{rho bounded} still holds in this case.

\begin{prop}\label{gradient0}
Assume $K=0$ in (\ref{condition}) and $(\phi^{'}(r))^2=\phi(r)\phi^{''}(r), \forall r\in [r_0,\bar r]$. Then flow (\ref{eflow0}) exists for all $t\in [0.\infty)$ and converges to a slice $\rho=c$ in $C^k, \forall k$ as $t\to \infty$.
\end{prop}

\begin{proof} Identity $(\phi^{'}(r))^2=\phi(r)\phi^{''}(r)$ is equivalent to $(\log \phi(r))^{''}=0$. That is,
\begin{equation}\label{phiab} \phi(r)=ae^{br}, \end{equation} for some constants $a>0, b\in \mathbb R$. In the proof of Theorem \ref{thmgradient}, at the critical point of $|\nabla \rho|^2$, (\ref{gradient evolution}) holds,
\begin{equation}
    \begin{array}[]{rll}
    \mathcal L(\frac{|\tilde\nabla \rho|^2}{2})
    = &-\frac{1}{\o}\sum_{j\ge2}\Big(\rho_{jj}+\frac{1}{2}\frac{\phi\phi'}{\o^2}|\tilde\n\rho|^2\Big)^2-\frac{1}{\o}\tilde Ric(\tilde \n\rho,\tilde \n\rho)\\
&+ \frac{|\tilde\n\rho|^2}{\o^5}\Big[n\frac{\phi''\phi-\phi'^2}{\phi^2}|\tilde\n\rho|^4\o^2+(n-1)(\phi''\phi-\phi'^2)|\tilde\n\rho|^2\o^2\\
&
  -(n+2)\phi'^2|\tilde\n\rho|^4-\frac{3}{4}(n-1)\phi^2\phi'^2|\tilde\n\rho|^2 \Big]
\end{array}
    \end{equation}
It follows that \[\frac{|\tilde\nabla \rho|^2}{2}\le \max_{t=0}\frac{|\tilde\nabla \rho|^2}{2}.\]
Therefore, by the standard theory of parabolic quasilinear PDE, we have regularity estimates in $C^k$ for all $k\ge 1$ and the flow exists all time.

To show the convergence, we only need to show $\|\tilde \nabla \rho\|_{L^{\infty}} \to 0$ as $t\to \infty$. First, if $b\neq 0$ in (\ref{phiab}), by the $C^0$ estimate, we have
\begin{equation}
    \mathcal L(\frac{|\tilde\nabla \rho|^2}{2})
    \le -\frac{|\tilde\n\rho|^2}{\o^5}\Big[
  (n+2)\phi'^2|\tilde\n\rho|^4+\frac{3}{4}(n-1)\phi^2\phi'^2|\tilde\n\rho|^2 \Big]\le -C(\frac{|\tilde\nabla \rho|^2}{2})^2,
    \end{equation}
    for some $C>0$.
By standard ODE comparison to the equation $f^{'}(t)=-Cf^2(t)$, we get \begin{equation}\label{gg0}
    \di\max_{M(t)} |\tilde\nabla \rho|^2\le \frac{\max_{M(0)}|\tilde\nabla \rho|^2}{Ct\max_{M(0)}|\tilde\nabla \rho|^2+1},
\end{equation}
for some $C>0$ which is independent of $t$.

If $b=0$ in (\ref{phiab}), then $\phi=a>0$ is a constant function. In this case, $H=-\frac{1}{a}\tilde\nabla(\frac{\tilde \nabla \rho}{\sqrt{a^2+|\tilde \nabla \rho|^2}})$.  Evolution equation (\ref{eflow0}) becomes
\[\rho_t=\tilde\nabla(\frac{\tilde \nabla \rho}{\sqrt{a^2+|\tilde \nabla \rho|^2}}).\]
Multiply $\rho$ in above equation, then integrate over $[0, t]\times B^n$,
\[\frac12(\int_{B^n}\rho^2(t,.)-\int_{B^n}\rho^2(0,.))=-\int_0^t\int_{B^n}\frac{|\tilde \nabla \rho(t,.)|^2}{\sqrt{a^2+|\tilde \nabla \rho(t,.)|^2}}dt.\]
The left hand side is bounded as $t\to \infty$, by regularity estimate, we must have
\begin{equation}\label{gg00}\int_{B^n}\frac{|\tilde \nabla \rho(t,.)|^2}{\sqrt{a^2+|\tilde \nabla \rho(t,.)|^2}}\to 0, \quad t\to \infty.\end{equation}
That is $\tilde \nabla \rho(t,.)\to 0$ in $L^2$, regularity estimates imply $\tilde \nabla \rho(t,.)\to 0$ in $L^{\infty}$.

In conclusion, evolution equation \eqref{eflow0} with $M$ as the initial data
has a smooth solution for $t\in [0,\infty)$ and the solution hypersurfaces converge to a level set of $r$ as $t\rightarrow \infty$. \end{proof}

As a consequence, in the case of $K=0$ in (\ref{condition}),  if $\O\subset \mathbf N^{n+1}$ is a domain bounded by a smooth graphical hypersurface $M$ and $S(r_0)$, then there exist a function $\xi$ such that
    \begin{align}
        Area(M)\ge \xi(Vol(\O)).
    \end{align}

\medskip

{\bf Acknowledgements.} The authors would like to thank Greg Galloway for the reference \cite{CVE} on photon surfaces and Frank Morgan for his comments
on the isoperimetric problem.
 

\begin{thebibliography}{99}
     \bibitem{Be} Besse, Arthur L. {\em Einstein manifolds. } Reprint of the 1987 edition. Classics in Mathematics. Springer-Verlag, Berlin, 2008.
     \bibitem{Br} Brendle, Simon. {\em Constant mean curvature surfaces in warped product manifolds,} Publications Mathématiques de l'IH\'ES 117, 247--269 (2013)
         \bibitem{Cant} Cant, Dylan, {\em A curvature flow and application to isoperimetric inequality}, preprint.
\bibitem{CVE} Claudel, Clarissa-Marie,  Virbhadra, K. S. and Ellis, G. F. R.
{\em The geometry of photon surfaces,}
J. Math. Phys. 42 (2001), no. 2, 818--838.
 \bibitem{CY} D. Christodoulou and S.-T. Yau,
Some remarks on the quasi-local mass. Mathematics and general relativity (Santa Cruz, CA, 1986), 9--14,
Contemp. Math., 71, Amer. Math. Soc., Providence, RI, 1988.

\bibitem{E} Engelstein, Max ; Marcuccio, Anthony; Maurmann, Quinn; Pritchard, Taryn
{\em Isoperimetric problems on the sphere and on surfaces with density.} New York J. Math. 15 (2009), 97--123.

     \bibitem{GL} Guan, Pengfei; Li, Junfang. {\em A new mean curvature type of flow in space forms}, International Mathematics Research Notices, 2015, no. 13, 4716-4740.
       \bibitem{H} Howe, Sean. {\em The log-convex density conjecture and vertical surface area in warped products}. Adv. Geom. 15 (2015), no. 4, 455--468.
     \bibitem{H1} Huisken, Gerhart. {\em Flow by mean curvature of convex surfaces into spheres}, J. Diff. Geom.  20 (1984), 237-266

     \bibitem{H2} Huisken, Gerhart. {\em Contracting convex hypersurfaces in Riemannian manifolds by their mean curvature}, Invent. Math. 84 (1986), no. 3, 463-480.

     \bibitem{H3} Huisken, Gerhart. {\em The volume preserving mean curvature flow}, J. Reine Angew. Math. 382 (1987), 35-48.

     \bibitem {LW} Li, Chunhe; Wang, Zhizhang. {\em An necessary condition for isoperimetric inequality in warped product space}, arXiv:1610.02223
     \bibitem{L} A. Lichnerowicz, G\'eom\'etrie des groupes de transformations, Travaux et Recherches Math\'ematiques, III. Dunod, Paris,
1958.

     \bibitem{M} Morgan, Frank. {\em The Log-convex Density Conjecture.}
      http://sites.williams.edu/Morgan/2010/04/03/the-log-convex-density-conjecture/

     \bibitem{R} Reilly, Robert C. {\em On the Hessian of a function and the curvatures of its graph,} Michigan Math. J. 20 (1973), 373-383.
     \bibitem{Ri} Ritor\'e, Manuel. {\em Constant geodesic curvature curves and isoperimetric domains in rotationally symmetric surfaces.} Comm. Anal. Geom. 9 (2001), no. 5, 1093--1138.

\bibitem{WWZ} Wang,  Mu-Tao, Wang, Ye-Kai Wang and Zhang, Xiangwen.  {\em Minkowski formulae and Alexandrov
theorems in spacetime}, J. of Differential Geom., Vol. 105, No. 2 (2017), 249--290.  arXiv: 1409.2190

\end{thebibliography}
  \end{document}